\documentclass[11pt, leqno]{amsart}
\usepackage{amsmath, amssymb, latexsym}
\usepackage{mathrsfs}
\usepackage{enumerate}
\usepackage{color}
\usepackage[colorlinks=true, pdfstartview=FitV, linkcolor=blue,%
citecolor=blue, urlcolor=blue]{hyperref}
\usepackage[all, knot]{xy}
\xyoption{arc}
\usepackage{tikz-cd}
\usetikzlibrary{cd}
\usepackage{tikz}
\usepackage{amsmath}
\usetikzlibrary{cd}
\usetikzlibrary{arrows.meta}

\numberwithin{equation}{section}

\newtheorem{theorem}{Theorem}[section]

\newtheorem{lemma}[theorem]{Lemma}
\newtheorem{proposition}[theorem]{Proposition}

\theoremstyle{definition}
\newtheorem{definition}[theorem]{Definition}

\newcommand{\Hom}{\operatorname{Hom}}

\newcommand{\C}{\mathbf{C}}

\newcommand{\Z}{\mathbf{Z}}
\newcommand{\h}{\mathfrak{h}}
\newcommand{\g}{\mathfrak{g}}

\title[Highest weight modules over Borcherds-Bozec superalgebras]
{Highest weight modules over Borcherds-Bozec superalgebras and their character formula}

\author[Zhaobing Fan]{Zhaobing Fan}
\address{Harbin Engineering University,
	Harbin, China}
\email{fanzhaobing@hrbeu.edu.cn}

\author[Jiaqi Huang]{Jiaqi Huang}
\address{Harbin Engineering University,
	Harbin, China}
\email{jiaqihuang@hrbeu.edu.cn}
\thanks{}

\author[Seok-Jin Kang]{Seok-Jin Kang}
\address{Korea Research Institute of Arts and Mathematics,
	Asan-si, Chungcheongnam-do, 31551, Korea}
\email{soccerkang@hotmail.com}
\thanks{}

\author[Yong-Su Shin]{Yong-Su Shin${}^*$}
\thanks{${}^*$The corresponding author}
\address{Department of Mathematics, Sungshin Women's University,  Seoul, 02844, Republic of  Korea}
\email{ysshin@sungshin.ac.kr}

\keywords{Borcherds-Bozec superalgebra, highest weight module, character formula}

\subjclass[2010] {17B37, 17B67, 16G20}

\begin{document}

\begin{abstract}
 We present and prove the Weyl-Kac type
 character formula for the irreducible highest weight modules over
Borcherds-Bozec superalgebras with dominant integral highest weights.
\end{abstract}

\maketitle

\section*{Introduction}

The {\it Kac-Moody algebras} were introduced independently by V. G. Kac and R. V. Moody
as an infinite dimensional generalization of finite dimensional semisimple
Lie algebras over $\C$ \cite{Kac68, Moody68}.
In \cite{Kac74}, Kac proved a character formula, called the {\it Weyl-Kac formula},
for the irreducible highest weight modules with dominant integral highest weights.
It turned out that the Weyl-Kac formula for {\it affine Kac-Moody algebras}
has significant applications to number theory,
combinatorics, mathematical physics, etc.

\vskip 2mm

In \cite{Bor88}, R. E. Borcherds generalized the notion of Kac-Moody algebras
to define {\it Borcherds algebras} associated with Borcherds-Cartan matrices.
These matrices have non-positive diagonal entries, which means the Borcherds algebras
may have simple roots with the norm $\le 0$, the {\it imaginary} simple roots.
Hence the index set $I$ has a decomposition $I = I^{\text{re}} \sqcup I^{\text{im}}$,
where $^{\text{re}} = \{i \in I \mid a_{ii} = 2\}$,
$I^{\text{m}} = \{i \in I \mid a_{ii} \le 0\}$.
He also generalized the Weyl-Kac formula for the case of Borcherds algebras.
A special case of Borcherds algebras, the {\it Monster Lie algebra},
played an important role in Borcherds' proof of the famous {\it  Moonshine conjecture} \cite{Bor92}.

\vskip 2mm

The {\it Borcherds-Bozec algebras} are a further generalization of Borcherds algebras
associated with Borcherds-Cartan matrices having higher degree positive and negative
simple root vectors.
The index set for these simple root vectors is denoted by
$I^{\infty}=(I^{\text{re}}\times \{1\}) \sqcup (I^{\text{im}}\times \Z_{\textgreater 0})$.
Borcherds-Bozec algebras arise as a natural algebraic structure relevant to the theory of
perverse sheaves on the representation varieties of quivers with loops
\cite{Bozec2014a, Bozec2014b}.
The Weyl-Kac type character formula for the irreducible highest weight modules was proved in \cite{BSV2016}.

\vskip 2mm

A {\it Lie superalgebra} is a $\Z_{2}$-graded vector space
$L = L_{\bar 0} \oplus L_{\bar 1}$ with the {\it Lie superbracket}
that satisfies the {\it super-symmetry} and  the {\it super-Jacobi identity}.

\vskip 2mm

The theory of Lie superalgebras was intensively developed by Kac \cite{Kac77}.
In \cite{Kac78}, Kac defined the notion of Kac-Moody superalgebras
and proved the super-version of Weyl-Kac formula.
Inspired by the idea of \cite{Kac78} and \cite{Bor88}, U. Ray proved the
Weyl-Kac formula for Borcherds superalgebras.

\vskip 2mm

In this paper, we introduce the notion of {\it odd Borcherds-Cartan datum}
and the associated {\it Borcherds-Bozec superalgebras}.
The highest weight representation theory is discussed as a mixture of
those over Kac-Moody superalgebras and Borcherds-Bozec algebras.
We also prove important commutation relations of {\it Casimir operator} and
higher degree simple root vectors.
Finally, combining all the above results together,
we present and prove the Weyl-Kac type character formula for the irreducible
highest weight modules over Borcherds-Bozec superalgebras with dominant integral
highest weights.

\vskip 3mm

\section{Odd Borcherds-Cartan Datum}

Let $I$ be an index set which can be countably infinite.
An integer-valued matrix $A=(a_{ij})_{i,j \in I}$ is called an {\it
even symmetrizable Borcherds-Cartan matrix} if it satisfies the following conditions:
\begin{itemize}
\item[(i)] $a_{ii}=2, 0, -2, -4, ...$,

\item[(ii)] $a_{ij}\le 0$ for $i \neq j$,

\item[(iii)] there exists a diagonal matrix $D=\text{diag} (d_{i} \in \Z_{>0} \mid i \in I)$ such that $DA$ is symmetric.
\end{itemize}

\vskip 2mm

Set $I^{\text{re}}=\{i \in I \mid a_{ii}=2 \}$,
$I^{\text{im}}=\{i \in I \mid a_{ii} \le 0\}$ and
$I^{\text{iso}}=\{i \in I \mid a_{ii}=0 \}$.

\vskip 2mm

{ An {\it odd Borcherds-Cartan datum}} consists of :

\begin{itemize}

\item[(a)] an even symmetrizable Borcherds-Cartan matrix $A=(a_{ij})_{i,j \in I}$,

\item[(b)] a free abelian group $P$, the {\it weight lattice},

\item[(c)] $P^{\vee} := \Hom(P, \Z)$, the {\it dual weight lattice},

\item[(d)] $\Pi=\{\alpha_{i} \in P  \mid i \in I \}$, the set of {\it simple roots},

\item[(e)] $\Pi^{\vee}=\{h_i \in P^{\vee} \mid i \in I \}$, the set of {\it simple coroots},

{ \item[(f)] a subset $I^{\text{odd}}$ of $I$, the set of {\it odd indices}},

\end{itemize}

\noindent satisfying the following conditions

\begin{itemize}

\item[(i)] $\langle h_i, \alpha_j \rangle = a_{ij}$ for all $i, j \in I$,

\item[(ii)] $\Pi$ is linearly independent over $\C$,

\item[(iii)] { for all $i \in I^{\text{re}} \cap I^{\text{odd}}$, we have $a_{ij} \in 2 \Z $},

\item[(iv)] for each $i \in I$, there exists an element $\Lambda_{i} \in P$ such that
$$\langle h_j , \Lambda_i \rangle = \delta_{ij} \ \ \ \text{for all} \ i, j \in I,$$

\end{itemize}

\vskip 2mm

{ We denote $I^{\text{even}} = I \setminus I^{\text{odd}}$, the set of {\it even indices}}.
Given an even symmetrizable Borcherds-Cartan matrix, it can be shown that such
{ an odd} Borcherds-Cartan datum always exists, which
is not necessarily unique. The $\Lambda_i$  $(i \in I)$ are called the {\it fundamental weights}.

\vskip 2mm

We define
\begin{equation} \label{eq:P}
P^{+}:=\{\lambda \in P \mid \langle h_i, \lambda \rangle \ge 0 \ \text{for} \ i \in I,
\ \ \langle h_{i}, \lambda \rangle \in 2 \Z \ \text{for} \ i \in I^{\text{re}} \cap I^{\text{odd}}  \}.
\end{equation}
The elements in $P^{+}$ are called the {\it dominant integral weighs}.

\vskip 2mm

The free abelian group $Q:= \bigoplus_{i \in I} \Z \, \alpha_i$ is called the {\it root lattice}.
Set $Q_{+}: = \sum_{i \in I} \Z_{\ge 0}\, \alpha_{i}$ and $Q_{-}: = -Q_{+}$.
For $\beta = \sum k_i \alpha_i \in Q\textbf{}_{+}$, we define its {\it height} to be
{ $\text{ht} (\beta):=\sum k_i$}.

\vskip 2mm

Let ${\mathfrak h} := \C \otimes_{\Z} P^{\vee}$ be the {\it Cartan subalgebra}.
We define
a partial ordering on ${\mathfrak h}^{*}$ by setting $\lambda \ge \mu$
if and only if $\lambda - \mu  \in Q_{+}$ for $\lambda, \mu \in {\mathfrak h}^{*}$.

\vskip 2mm

 We decompose ${\mathfrak h} = {\mathfrak h}' \oplus {\mathfrak h}''$, where
${\mathfrak h}' = \bigoplus_{i \in I} \C \, h_{i}$ and ${\mathfrak h}''$ is its complementary subspace.
Since $A$ is symmetrizable, we can define
a symmetric bilinear form $( \ \, , \ \, )$ on ${\mathfrak h}$ by
\begin{equation} \label{eq:h-bilinear}
\begin{aligned}
& (h, h_i) = d_{i}^{-1} \langle h, \alpha_ {i} \rangle \ \ \text{for} \ h \in {\mathfrak h}, \\
& (h', h'') = 0 \  \ \text{for} \ h', h'' \in {\mathfrak h}''.
\end{aligned}
\end{equation}

{In particular}, we have
$$(h_i, h_j) = d_{j}^{-1} a_{ij} = d_{i}^{-1} a_{ji} = (h_j, h_i).$$

\vskip 2mm

Moreover, since $\Pi$ is linearly independent,
it is straightforward to verify that $( \ \, , \ \, )$ is non-degenerate on ${\mathfrak h}$.
Hence there is a linear isomorphism $\nu: \mathfrak{h} \rightarrow \mathfrak{h}^{*}$
defined by
\begin{equation} \label{eq:linear_iso}
\nu(h)(h') = (h, h') \ \ \text{for} \  h, h' \in {\mathfrak h},
\end{equation}
{ which implies}
$$\nu(h_i) = d_{i}^{-1} \alpha_{i}, \ \ \nu^{-1}(\alpha_i) = d_i h_i.$$

Then we obtain the induced non-degenerate symmetric bilinear for $(\ \, , \ \, )$
on ${\mathfrak h}^{*}$ defined by
\begin{equation} \label{eq:hstar}
(\lambda, \mu): = (\nu^{-1}(\lambda), \mu^{-1}(\mu)) \ \ \text{for} \ \lambda, \mu \in {\mathfrak h}^{*}.
\end{equation}
In particular, we have
\begin{equation} \label{eq:alpha_i}
(\alpha_i, \lambda) = d_{i} \langle h_{i}, \lambda \rangle \ \ \text{for} \ i \in I.
\end{equation}

\vskip 2mm

For each  $i \in I^{\text{re}}$, we define the {\it simple reflection}
$r_{i}:{\mathfrak h}^{*} \rightarrow {\mathfrak h}^{*}$ by
\begin{equation} \label{eq:simple_reflection}
r_{i}(\lambda)= \lambda - \langle h_{i}, \lambda \rangle \,  \alpha_{i}
\ \ \text{for} \ \lambda \in {\mathfrak h}^{*}.
\end{equation}

The subgroup $W$ of $GL({\mathfrak h}^{*})$ generated by the simple reflections $r_{i}$ $(i \in I^{\text{re}})$ is called the {\it Weyl group} of the { odd} Borcherds-Cartan datum given above.
It is easy to check that $(\ \,  , \ \,  )$ is $W$-invariant.

\vskip 2mm

For the rest of the paper, we fix a linear functional $\rho \in {\mathfrak h}^{*}$ satisfying
\begin{equation} \label{eq:rho}
(\rho, \alpha_{i}) = \ \dfrac{1}{2} (\alpha_{i}, \alpha_{i}) \ \ \text{for all} \ i \in I.
\end{equation}

\vskip 3mm

\section{Borcherds-Bozec Superalgebras}

In this section, we define the notion of {\it Borcherds-Bozec superalgebras}.
We first define a new index set
\begin{equation} \label{eq:Iinfty}
I^{\infty}:= (I^{\text{re}} \times \{1\}) \sqcup (I^{\text{im}}
\times \Z_{>0}).
\end{equation}

 For simplicity, we will often write $i$ for $(i,1)$.

 \vskip 2mm

 \begin{definition} \label{def:BBsuper}
 {\rm
The {\it Borcherds-Bozec superalgebra} associated with
 an { odd Borcherds-Bozec Cartan datum $(A, P, P^{\vee},
 \Pi,\Pi^{\vee}, I^{\text{odd}})$} is the Lie superalgebra
 ${\mathfrak g} = {\mathfrak g}_{\bar 0} \oplus {\mathfrak g}_{\bar 1}$
 over $\C$ generated by the elements $e_{il}$, $f_{il}$ $((i,l)\in I^{\infty})$ and ${\mathfrak h}$
 with defining relations
 \begin{equation} \label{eq:BBsuper}
  \begin{aligned}
  & [h, h'] = 0 \ \ \text{for} \ h, h' \in {\mathfrak h}, \\
  & [e_{il}, f_{jk}] = l  \delta_{ij} \delta_{kl}  h_{i} \ \ \text{for} \ i, j \in I, k, l \in \Z_{>0}, \\
  & [h, e_{jl}] = l  \langle h, \alpha_{j} \rangle e_{jl}, \ [h, f_{jl}] = -l \langle h, \alpha_{j} \rangle f_{jl}, \\
  & (\text{ad} e_{i})^{1 - l a_{ij}}  (e_{jl}) = 0 \ \ \text{for} \ i \in I^{\text{re}}, i \neq (j,l), \\
  & (\text{ad} f_{i})^{1 - l a_{ij}}  (f_{jl}) = 0 \ \ \text{for} \ i \in I^{\text{re}}, i \neq (j,l), \\
  & [e_{il},e_{jk}] = [f_{il}, f_{jk}] = 0 \ \ \text{for} a_{ij} = 0.
  \end{aligned}
 \end{equation}
 Here, $h \in {\mathfrak g}_{\bar 0}$, $e_{il}, f_{il}  \in {\mathfrak g}_{\bar 0}$ for $i \in I^{\text{even}}$
 and   $e_{il}, f_{il} \in {\mathfrak g}_{\bar 1}$ for $i \in I^{\text{odd}}$.

 }
 \end{definition}

 We will denote by $|x| = \bar 0, \bar 1$ for homogeneous elements $x \in {\mathfrak g}$.
 Hence $|h| = \bar 0$, $|e_{il}| = |f_{il}| =  \bar 0$ if $i \in I^{\text{even}}$
 and $|e_{il}| = |f_{il}| = \bar 1$ if $i \in I^{\text{odd}}$.

 \vskip 2mm

 For an element $\alpha \in Q$, define
 $${\mathfrak g}_{\alpha} = \{x \in {\mathfrak g} \mid
 [h, x] = \langle h, \alpha \rangle x \ \ \text{for all} \ h \in {\mathfrak h} \}.$$
 Then ${\mathfrak g}$ has the {\it root space decomposition}
 $${\mathfrak g} = \bigoplus_{\alpha>0} {\mathfrak g}_{-\alpha}
 \oplus {\mathfrak h} \oplus \bigoplus_{\alpha>0} {\mathfrak g}_{\alpha}.$$

 If ${\mathfrak g} \neq 0$, $\alpha$ is called a {\it root} of ${\mathfrak g}$.
 Let $\Delta$ be the set of roots of $\mathfrak g$.
 We denote by $\Delta_{+}$ (respectively, $\Delta_{-}$) the set of
 positive (respectively, negative) roots.
 It can be shown that $\dim {\mathfrak g}_{\alpha} = \dim {\mathfrak g}_{-\alpha}$.
 A root is {\it real} if $(\alpha, \alpha)>0$ and is {\it imaginary} if $(\alpha, \alpha) \le 0$.
We use the notation $\Delta^{\text{re}}$ and $\Delta^{\text{im}}$
 for the set of real roots and imaginary roots, respectively.

 \vskip 3mm

 By setting
 \begin{equation} \label{eq:degree}
 \text{deg}(h) = 0, \ \ \text{deg}(e_{il}) = l \alpha_{i}, \ \ \text{deg}(f_{il}) = -l \alpha_{i},
 \end{equation}
 every positive root can be written as
 \begin{equation*}
\alpha  =  s_{1} \alpha_{i_1} + \cdots + s_{r} \alpha_{i_r}
 = \alpha_{j_1} + \cdots +{ \alpha_{j_t}} \ (s_k>0, i_k, j_l \in I).
 \end{equation*}


A positive root $\alpha$ is {\it even} (respectively, {\it odd})
if the number of odd indices $j_l$ in $\alpha$ is even (respectively, odd).
Thus we get
\begin{equation} \label{eq:even}
{\mathfrak g}_{\bar 0} = \bigoplus_{\alpha \in \Delta_{+}^{\text{even}}} {\mathfrak g}_{\pm \alpha}
\ \ \ \text{and}
\ \ \ {\mathfrak g}_{\bar 1} = \bigoplus_{\alpha \in \Delta_{+}^{\text{odd}}} {\mathfrak g}_{\pm \alpha},
\end{equation}
 where $\Delta_{+}^{\text{even}}$ (respectively, $\Delta_{+}^{\text{odd}}$) denotes the set of
 positive even (respectively, odd) roots.

 \vskip 2mm

 \begin{proposition} \label{prop:g-bilinear}

 {\rm
 Let ${\mathfrak g} = \bigoplus_{\alpha \in \Delta} {\mathfrak g}_{\alpha}$ be a
 Borcherds-Bozec superalgebra.

 \vskip 2mm

 Then the non-degenerate symmetric bilinear form $( \ \, , \, \ )$ on ${\mathfrak h}$ can be extended
 to a bilinear form on ${\mathfrak g}$ satisfying the following  properties.

 \vskip 2mm

 \begin{itemize}

 \item[(a)]  $({\mathfrak g}_{\alpha}, {\mathfrak g}_{\beta}) = 0$ if $\alpha + \beta \neq 0$.

\vskip 2mm

 \item[(b)] $( \ \, , \ \, )$ is non-degenerate on ${\mathfrak g}$.

 \vskip 2mm

 \item[(c)] $( \ \, , \ \, )$ is {\it super-symmetric}; i.e.,
 $$(x, y) = -(-1)^{|x||y|} (y,x) \ \ \text{for all homogeneous elements} \ x, y \in {\mathfrak g}.$$

 \item[(d)] $(\ \, , \ \,)$ is {\it invariant}; i.e.,
 $$([x, y], z) = (x, [y, z]) \ \ \text{for all} \ x, y, z \in {\mathfrak g}.$$

 \item[(e)] For $\alpha \in \Delta_{+}$ and $x \in {\mathfrak g}_{\alpha}$, $y \in {\mathfrak g}_{-\alpha}$,
 we have
 $$[x, y] = (x, y) \nu^{-1}(\alpha).$$

 \end{itemize}

 }
 \end{proposition}

 \vskip 2mm

  \begin{proof}
  For each $N \in \Z_{> 0}$, let $\g(N)$ (respectively, $\g(-N)$) be the subspace of $\g$
  spanned by the super-brackets
  $[e_{i_1, l_1}, \ldots, e_{i_r, l_r}]$ (respectively,
  $[f_{i_1, l_1}, \ldots, f_{i_r, l_r}]$)
  such that  $ l_1 + \cdots + l_r = N.$
  Then $\g$ has a $\Z$-graded decomposition
  $$\g = {\mathfrak h} \oplus  \bigoplus_{N \in \Z} \g(N).$$

  We first extend the symmetric bilinear form $(\ \, , \ \,)$ on ${\mathfrak h}$ to
  to a super-symmetric bilinear form on $\g(1)$ by
  \begin{equation*}
  \begin{aligned}
(e_{il}, f_{jk}) &=\delta_{ij}\delta_{lk}d_{i}^{-1} \ \ \text{for all} \  i, j \in I, \ \ k, l \in \Z_{>0}, \\
(e_{il},h) & =(f_{il},h)=(e_{il},e_{il})=(f_{il},f_{il})=0.
\end{aligned}
\end{equation*}

 \noindent
 Then the bilinear form $(\ \, ,\ \, )$ on $\g(1)$ satisfies $\text{(d)}$ as long as both $[x,y]$ and $[y,z]$
    belong to  $\g_(1)$.
 To see this, it suffices to verify that
  $$ ([e_{i1},f_{i1}],h)=(e_{i1},[f_{i1},h])\ \text{for}\  h\in\h,$$
which follows from the definition of $\nu$.

 \vskip 2mm

 Now using the induction on $N$,  we would like to extend $(\ \, , \ \, )$ to a bilinear form on $\g(N)$
 satisfying

 \begin{itemize}

 \item[(i)] $(\g_{i}, \g_{j}) = 0$ if $|i|, |j| \le N$, $i + j \neq 0$,

 \vskip 2mm

 \item[(ii)] the statement (d) holds as long as $[x, y]$ and $[y,z]$ belong to $\g(N)$.

 \end{itemize}

 \vskip 2mm

 Assume that we already have the bilinear form on $\g(N-1)$ satisfying (i) and (ii).
 It suffices to define $(x, y)$ for $x \in \g(\pm N)$ and $y \in \g(\mp N)$.
 Note that $y$ can be expressed as a sum  $y =\sum_{i}  [u_i, v_i]$, where
 $u_{i}, v_{i}$ are homogeneous elements of non-zero degree in $\g(N-1)$.
 Then $[x, u_i] \in \g(N-1)$ and we may define
 $$(x, y) = \sum_{i} ([x, u_i], v_i).$$

 Actually, this is well-defined; if $i, j, s, t \in \Z$ are integers such that
 \begin{equation*}
 \begin{aligned}
 & |i + j| = |s + t| = N, \ i+j+s+t = 0, \
 |i|, |j|, |s|, |t| < N, \\
 & x_{i} \in \g(i), x_{j} \in \g(j), x_{s} \in \g(s), x_{t} \in \g(t),
 \end{aligned}
 \end{equation*}
 then we have
 $$([[x_i, x_j], x_s], x_t) = (x_i, [x_j [x_s, x_t]]),$$
 which gives the well-definedness.

 \vskip 2mm

 Hence we have constructed a bilinear form satisfying (a) and (d).

 \vskip 2mm

 Let $\alpha \in \Delta_{+}$ and $x \in \g_{\alpha}$, $y \in \g_{-\alpha}$.
 Then for any $h \in \h$, we have
 $$
 ([x, y] - (x, y) \nu^{-1}(\alpha), h) = (x, [y, h]) - (x, y) \langle h, \alpha \rangle =0.
 $$
 Thus the statement (e) holds.

 \vskip 2mm

 The statement (c) follows from the definition.
 \end{proof}

 \vskip 2mm

 Let $\alpha \in \Delta_{+}$ be a positive root.
 Choose dual bases $\{e_{\alpha}^{(k)} \} _{k=1}^{d_{\alpha}}$ of ${\mathfrak g}_{\alpha}$
 and  $\{e_{-\alpha}^{(k)} \} _{k=1}^{d_{\alpha}}$ of ${\mathfrak g}_{-\alpha}$
 with respect to $(\ \, , \ \,)$, where $d_{\alpha} = \dim {\mathfrak g}_{\alpha}$.
 That is,
 $$(e_{\alpha}^{(k)}, e_{-\alpha}^{(l)}) = \delta_{kl}, \ \ \
 (e_{-\alpha}^{(l)}, e_{\alpha}^{(k)}) = -\delta_{kl}.$$

 Hence for $x \in {\mathfrak g}_{\alpha}$ and $y \in {\mathfrak g}_{-\alpha}$, we have
 \begin{equation*}
 \begin{aligned}
& x  = \sum_{k=1}^{d_{\alpha}} (x, e_{-\alpha}^{(k)}) e_{\alpha}^{(k)}, \ \ \
y  = \sum_{k=1}^{d_{\alpha}} (e_{\alpha}^{(k)}, y) e_{-\alpha}^{(k)}, \\
& (x, y)  = \sum_{k=1}^{d_{\alpha}} (x, e_{-\alpha}^{(k)}) (e_{\alpha}^{(k)}, y).
 \end{aligned}
 \end{equation*}

 \vskip 2mm

 \begin{lemma}\label{1}
 {\rm
 If $\alpha, \beta \in \Delta$ and $x \in {\mathfrak g}_{\beta - \alpha}$,
 for any homogeneous element $z \in {\mathfrak g}$, we have

 \begin{itemize}

 \item[(a)] $\sum_{k=1}^{d_{\alpha}} e_{-\alpha}^{(k)} \otimes [z, e_{\alpha}^{(k)}]
 = (-1)^{|z|} \sum_{l=1}^{d_{\beta}} [e_{-\beta}^{(l)} \otimes e_{\beta}^{(l)}$
 in ${\mathfrak g} \otimes {\mathfrak g}$,

 \vskip 2mm

 \item[(b)] $\sum_{k=1}^{d_{\alpha}} [e_{-\alpha}^{(k)} [z, e_{\alpha}^{(k)}]]
 = (-1)^{|z|} \sum_{l=1}^{d_{\beta}}[[e_{-\beta}^{(l)}, z], e_{\beta}^{(l)}]$
 in ${\mathfrak g}$,

 \vskip 2mm

  \item[(c)] $\sum_{k=1}^{d_{\alpha}} e_{-\alpha}^{(k)} [z, e_{\alpha}^{(k)}]
 = (-1)^{|z|} \sum_{l=1}^{d_{\beta}}[e_{-\beta}^{(l)}, z]e_{\beta}^{(l)}$
 in $U({\mathfrak g})$.

 \end{itemize}

 }
 \end{lemma}

 \begin{proof}
  We define the bilinear form $(\ ,\ )$ on $\g\otimes\g $ by $(a\otimes b,c\otimes d)=(a,c)(b,d)$.
  For  $e\in \g_{\alpha}, \ f\in \g_{-\beta}$, we have
  \begin{equation}
  \begin{aligned}
  (\sum_{s}e_{-\alpha}^{s}\otimes [z,e_{\alpha}^{s}],e\otimes f) &=\sum_{s}(e_{-\alpha}^{s},e)([z,e_{\alpha}^{s}],f)\\
  &=(-1)^{g(\alpha)g(\beta)+1}(e,[z,f]).
  \end{aligned}
  \end{equation}

  \begin{equation}
  \begin{aligned}
  (\sum_{s}[e_{-\beta}^{s},z]\otimes e_{\beta}^{s},e\otimes f) &=\sum_{s}([e_{-\beta}^{s},z],e)(e_{\beta}^{s},f)\\
  &=(-1)^{g(\alpha)g(\beta)+1+g(z)}(e,[z,f]).
  \end{aligned}
  \end{equation}

Thus $\sum_{s}e_{-\alpha}^{s}\otimes [z,e_{\alpha}^{s}]=(-1)^{|z|}\sum_{s}[e_{-\beta}^{s},z]\otimes e_{\beta}^{s}$ in $\g \otimes \g$.

\vskip 2mm

  Applying  (a) to the linear maps
  $$\g \otimes \g \rightarrow \g, \ \ x \otimes y \mapsto [x, y]$$
  and
  $$\g \otimes \g \rightarrow U(\g), \ \ x \otimes y \mapsto xy.$$
  we obtain (b) and (c).
 \end{proof}

 \vskip 3mm

 \section{Representation Theory and Casimir Operator}

 Let ${\mathfrak g}$ be a Borcherds-Bozec superalgebra.
 The category ${\mathscr M}$ consists  of $\g$-modules $M$
 satisfying the following properties:

 \begin{itemize}

 \item[(i)] $M = \bigoplus_{\lambda \in {\mathfrak h}^*} M_{\lambda}$, where
 $M_{\lambda} = \{ v \in M \mid h v = \langle h, \lambda \rangle {\color{red} v}
 \ \text{for all} \ h \in {\mathfrak h} \}$,

 \vskip 2mm

 \item[(ii)] $\dim M_{\lambda} < \infty$ for all $\lambda \in {\mathfrak h}^*$,

 \vskip 2mm

 \item[(iii)] there exist finitely many $\lambda_1, \ldots, \lambda_s \in {\mathfrak h}^*$
 such that
 $$\text{wt}(M) \subset \bigcup_{j=1}^{s} (\lambda_j  - Q_{+}), \ \ \text{where} \ \
 \text{wt}(M) = \{\lambda \in {\mathfrak h}^* \mid M_{\lambda} \neq 0\}.$$

\end{itemize}

 \vskip 2mm

 For a $\g$-module $M = \bigoplus_{\lambda \in {\mathfrak h}^{*}} M_{\lambda}$
 in the category $\mathscr{M}$, we define its {\it character} to be
 $$\text{ch} M = \sum_{\lambda \in {\mathfrak h}^{*}} (\dim V_{\lambda}) e^{\lambda},$$
 where  $e^{\lambda}$ $(\lambda \in {\mathfrak h}^{*})$ denotes the additive basis of
 $\C[{\mathfrak h}^{*}]$.

 \vskip 2mm

 Let $\lambda \in \text{wt}(M)$.
 Since $e_{il} M_{\lambda} \subset M_{\lambda + l \alpha_{i}}$, due to the condition (iii),
 it is easy to see that $\g_{\alpha} M_{\lambda} = 0$ for all but finitely many
 $\alpha \in \Delta_{+}$.

 \vskip 2mm

 We define an element in $U(\g)$ by
 \begin{equation} \label{eq:Omega0}
 \Omega_{0} = 2 \sum_{\alpha \in \Delta_{+}}
 \sum_{k=1}^{d_{\alpha}} e_{-\alpha}^{(k)} e_{\alpha}^{(k)},
 \end{equation}
 where $\{e_{\alpha}^{(k)}\}_{k=1}^{d_{\alpha}}$ and
 $\{e_{-\alpha}^{(k)}\}_{k=1}^{d_{\alpha}}$ are dual bases of $\g_{\alpha}$
 and $\g_{-\alpha}$ for $\alpha \in \Delta_{+}$.
 Then for any module $M$ in the category ${\mathscr M}$, the operator
 $\Omega_{0}$ is well-defined.

 \vskip 2mm

 The {\it Casimir operator} $\Omega$ on $M= \bigoplus_{\lambda \in {\mathfrak h}^{*}} M_{\lambda}$
 is defined as follows:
 \begin{equation} \label{eq:Casimir}
 \Omega(v) = (\lambda + 2 \rho, \lambda) v + \Omega_{0}(v)
 \ \ \text{for} \ v \in M_{\lambda}.
 \end{equation}

 \vskip 2mm

 \begin{proposition} \label{prop:commutation}
 {\rm
 Let $M = \bigoplus_{\lambda \in {\mathfrak h}^{*}} M_{\lambda}$ be a $\g$-module
 in the category ${\mathscr M}$ and let $v \in M_{\lambda}$.

 \vskip 2mm

 Then the Casimir operator satisfies the following commutation relations:
 \begin{equation} \label{eq:commutation}
 \begin{aligned}
& [\Omega, f_{i}] (v)  = 0 \ \ \text{for} \ i \in I^{\text{re}}, \\
& [\Omega, f_{il}](v)  = (l^2 - l) (\alpha_i, \alpha_i) f_{il} v
 \ \ \text{for} \ i \in I^{\text{im}}, l \in \Z_{>0}.
 \end{aligned}
 \end{equation}

 }

 \end{proposition}

 \begin{proof}
  For any $(i,l)\in I^{\infty}$, we have
  \begin{align*}
  \sum_{\alpha\in\Delta_{+}}\sum_{k=1}^{d_{\alpha}}[e_{-\alpha}^{k}e_{\alpha}^{k},f_{il}] & =\sum_{\alpha\in\Delta_{+}}\sum_{k=1}^{d_{\alpha}}(e_{-\alpha}^{k}e_{\alpha}^{k}f_{il}-f_{il}e_{-\alpha}^{k}e_{\alpha}^{k})\\
   &=\sum_{\alpha\in\Delta_{+}}\sum_{k=1}^{d_{\alpha}}(e_{-\alpha}^{k}[e_{\alpha}^{k},f_{il}]-[f_{il},e_{-\alpha}^{k}]e_{\alpha}^{k}).\\
  \end{align*}

  Since $$\sum_{\alpha\in\Delta_{+}}\sum_{k=1}^{d_{\alpha}}e_{-\alpha}^{k}e_{\alpha}^{k}f_{il}=\sum_{k=1}^{d_{l\alpha_i}}e_{-l\alpha_{i}}^{k}[e_{l\alpha_{i}}^{k},f_{il}]+\sum_{\alpha\in\Delta_{+}-\{l\alpha_{i}\}}\sum_{k=1}^{d_{\alpha}}e_{-\alpha}^{k}[e_{\alpha}^{k},f_{il}],$$

  $$\sum_{k}e_{-l\alpha_{i}}^{k}[e_{l\alpha_{i}}^{k},f_{il}]=d_{i}lf_{il}h_{i},$$
  by Lemma\ \ref{1}, we obtain
  \begin{align*}
  \sum_{\alpha\in\Delta_{+}-\{l\alpha_{i}\}}\sum_{k}e_{-\alpha}^{k}[e_{\alpha}^{k},f_{il}] &=(-1)^{|\alpha|
  |l\alpha_{i}|+1}\sum_{\alpha\in\Delta_{+}-\{l\alpha_{i}\}}\sum_{k}e_{-\alpha}^{k}[f_{il},e_{\alpha}^{k}]\\
  &=(-1)^{|\alpha| |l\alpha_{i}|+|l\alpha_{i}|+1}\sum_{\alpha\in\Delta_{+}-\{l\alpha_{i}\}}\sum_{k}[e_{-\alpha+l\alpha_{i}}^{k},f_{il}]e_{\alpha-l\alpha_{i}}^{k}\\
  &=\sum_{\alpha\in\Delta_{+}-\{l\alpha_{i}\}}\sum_{k}[f_{il},e_{-\alpha+l\alpha_{i}}^{k}]e_{\alpha-l\alpha_{i}}^{k}\\
  &=\sum_{\beta\in\Delta_{+}}\sum_{k}[f_{il},e_{-\beta}^{k}]e_{\beta}^{k},
  \end{align*}
  which yields
  $$[\Omega_{0},f_{il}]=2d_{i}lf_{il}h_{i}.$$
  Therefore we obtain
  \begin{align*}
  [\Omega,f_{il}]v &=(\lambda-l\alpha_{i}+2\rho,\lambda-l\alpha_{i})f_{il}v-(\lambda+2\rho,\lambda)f_{il}v+2(l\alpha_{i},\lambda)f_{il}v\\
  &=(l^{2}-l)(\alpha_{i},\alpha_{i})f_{il}v.
  \end{align*}
  \end{proof}

 \vskip 2mm

 A $\g$-module $V$ is called a {\it highest weight module} with
 highest weight $\lambda$ if there exists a non-zero vector $v_{\lambda} \in V$
 such that

 \begin{itemize}

 \item[(i)] $h v_{\lambda} = \langle h, \lambda \rangle v_{\lambda}$ for all $h \in {\mathfrak h}$,

 \item[(ii)] $e_{il} v_{\lambda} = 0$ for all $(i,l) \in I^{\infty}$,

 \item[(iii)] $V = U(\g) v_{\lambda}$.

 \end{itemize}

 \vskip 2mm

 Such a vector $v_{\lambda}$ is called a {\it highest weight vector} with highest weight $\lambda$.
 Note that a highest weight module $V$ has a decomposition
 $V = \bigoplus_{\mu \le \lambda} V_{\mu}$.

 \vskip 2mm

For $\lambda \in {\mathfrak h}^{*}$, let $J(\lambda)$ be the left ideal of $U(\g)$
generated by the elements $h - \langle h, \lambda \rangle \mathbf{1}$ for $h \in {\mathfrak h}$,
and $e_{il}$ for $(i,l) \in I^{\infty}$.
Then $M(\lambda): = U(\g) \big/ J(\lambda)$ is a highest weight $\g$-module
with highest weight $\lambda$, called the {\it Verma module}.
Since $M(\lambda)$ is a free $U(\g^{-})$-module
generated by $v_{\lambda} = \mathbf{1} + j(\lambda)$,
every highest weight $\g$-module with highest weight $\lambda$ is a homomorphic
image of $M(\lambda)$.

\vskip 2mm

It can be shown that $M(\lambda)$ has a unique maximal submodule $R(\lambda)$.
We denote by $V(\lambda):= M(\lambda) \big/ R(\lambda)$ the irreducible
quotient of $M(\lambda)$.

 \vskip 2mm

 \begin{lemma}
 {\rm
 For a weight $\mu$ of $V(\lambda)$  with $\lambda \in P^{+}$, set
 $$W \circ \mu = \{w(\mu + \rho) - \rho \mid w \in W\}.$$

 Let $\tau$ be the element in $W \circ \mu$  such that $\text{ht}(\lambda- \tau)$ is minimal.
 Then
 $$(\tau + \rho, \alpha_i) \ge 0 \ \ \text{for all} \ i \in I^{\text{re}}.$$

 }
 \end{lemma}

\begin{proof}
Suppose $(\tau + \rho, \alpha_{i})<0$ for some $i \in I^{\text{re}}$ and set
$$\nu = r_{i}(\tau + \rho) - \rho = \tau - \langle h_i, \tau + \rho \rangle \alpha_i  \in W \circ \mu.$$
Then it is easy to see that $$\lambda - \nu = \lambda - \tau
+ \langle h_i, \tau + \rho \rangle \alpha_{i},$$
which implies $\text{ht}(\lambda - \nu) < \text{ht}(\lambda - \tau)$, a contradiction.
\end{proof}

 \vskip 3mm

 \section{The Character Formula}

 In this section, we give a character formula for the irreducible highest weight module
 $V(\lambda)$ with highest weight $\lambda \in P^{+}$.

 \vskip 2mm

 Let
 \begin{equation} \label{eq:R}
 R = \dfrac{\prod_{\alpha \in \Delta_{+}^{\text{eve}}} (1 - e^{-\alpha})^{\dim \g_{\alpha}}   }
 {\prod_{\alpha \in \Delta_{+}^{\text{odd}}} (1 + e^{-\alpha})^{\dim \g_{\alpha}}  }.
 \end{equation}

 By Poincar\'e-Birkhoff-Witt Theorem for Lie superalgebras \cite{Kac77},
we have
\begin{equation} \label{eq:PBW}
(e^{\rho} R) \text{ch} M(\lambda) = e^{\lambda+\rho}.
\end{equation}

\vskip 2mm

A similar argument in \cite[Lemma 1.6] {Kac78} gives
 \begin{lemma} \label{lem:wR}
 {\rm
 For all $w \in W$, we have
 \begin{equation} \label{eq:wR}
w (e^{\rho} R) = \epsilon(w) (e^{\rho} R).
 \end{equation}

 }
 \end{lemma}

 \vskip 2mm

 For  $\g$-module $M$ in the category $\mathscr{M}$,
 we say that $m \in M_{\mu}$ is a {\it primitive vector} if
 there exists a submodule $U \subset M$ such that
 \begin{equation} \label{eq:primitive}
m \notin U, \  \ e_{il} \in U \ \ \text{for all}  \  (i,l) \in I^{\infty}.
\end{equation}

  In this case, $\mu$ is called a {\it primitive weight} of $M$.

  \vskip 2mm

\begin{proposition} \label{prop:Omega0}
{\rm
Let $M(\lambda)$ be the Verma module with highest weight $\lambda \in P^{+}$.

\vskip 2mm

If $m \in M(\lambda)_{\mu}$ is a primitive vector, then we have
\begin{equation} \label{eq:Omega0}
\Omega_{0}(m) = 0.
\end{equation}
}
\end{proposition}

\begin{proof}
Let $m = \prod_{k=1}^{r} f_{i_k, s_k} \in M(\lambda)_{\mu} v_{\lambda}$  be a primitive vector.
 We claim that
 $$\Omega_{0}(m) = C m \ \  \text{for some} \ C \in \C.$$

If $r = 1$,  then $\Omega_{0}(f_{i_1, s_1} v_{\lambda})
 = 2 d_{i_1} s_1 \langle h_{i_1}, \lambda \rangle f_{i_1, s_1} v_{\lambda}$.
 Hence our claim is true.

 \vskip 2mm

 Assume that our assertion is true for the monomials of height $(\le r-1)$.
 Then the following calculation yields
 \begin{equation*}
 \begin{aligned}
 \Omega_{0}(m) & = [\Omega_{0}, f_{i_1, s_1}] \prod_{k=2}^{r} f_{i_k, s_k} v_{\lambda}
 + f_{i_1, s_1} \Omega_{0}(\prod_{k=2}^r f_{i_k, s_k} v_{\lambda}) \\
 & =  (s_{1}^{2}-s_{1}) (\alpha_{i_1},\alpha_{i_1}) \prod_{k=1}^r f_{i_k, s_k} v_{\lambda}
 + C \prod_{k=1}^r f_{i_k, s_k} v_{\lambda}\\
 & = ((s_{1}^2 - s_{1}) (\alpha_{i_1}, \alpha_{i_1}) + C) \prod_{k=1}^r f_{i_k, s_k} v_{\lambda},
 \end{aligned}
 \end{equation*}
 as desired.

 \vskip 2mm

 Since $m$ is a primitive vector, there exists a submodule $U \subset M(\lambda)$
 satisfying \eqref{eq:primitive}.
 Note that
 $$\Omega_0(m) = 2 \sum_{\alpha \in \Delta_{+}} \sum e_{-\alpha}^{(k)} e_{\alpha}^{(k)} m \in U.$$
 But
$$\Omega_{0}(m) = C m \notin U \ \ \text{unless} \ \ C=0,$$
 from which we deduce our assertion.
\end{proof}

  \vskip 2mm

  Let $M(\lambda)$ be the Verma module with highest weight $\lambda \in P^{+}$
  and let $m = \prod_{k=1}^{r} f_{i_k, s_k} \in M(\lambda)_{\mu}$ be a primitive vector.

  \vskip 2mm

Write
$$\lambda - \mu = \beta = \sum_{k=1}^{r} s_{k} \alpha_{i_k}
= \sum_{i \in I} a_{i} \alpha_{i} \ \ (a_{i}>0).$$

Assume that $(\mu + \rho, \alpha_{i}) \ge 0$ for all $i \in I^{\text{re}}$.
(For simplicity, we will often write $\mu + \rho  \in C$, where $C$ denotes the
{\it Weyl chamber}.)
Then
$$\beta - \alpha_i = \sum_{j \neq i} a_j \alpha_j + (a_{i}-1) \alpha_i.$$
In particular, if $i \in I^{\text{im}}$, we have
$$(\beta - \alpha_i, \alpha_i)
= \sum_{j \neq i} a_j (\alpha_j , \alpha_i)
+ (a_i - 1) (\alpha_i, \alpha_i) \le 0.$$

Since $m$ is a primitive vector of weight $\mu$, we have
\begin{equation*}
\Omega(m) = (\mu + 2 \rho, \mu) m
= ((\lambda + 2 \rho, \lambda) + \sum_{k=1}^{r} s_k (s_k - 1) (\alpha_{i_k}, \alpha_{i_k}))\, m,
\end{equation*}
 which implies
 $$(\mu + 2 \rho, \mu) - (\lambda + 2 \rho, \lambda)
 = \sum_{k=1}^{r} s_k (s_k -1) (\alpha_{i_k}, \alpha_{i_k}).$$

If $i \in I^{\text{re}}$,  then
\begin{equation*}
(\alpha_i, \mu + 2 \rho) = (\alpha_i, \mu + \rho) + \frac{1}{2} (\alpha_i, \alpha_i) >0,
\end{equation*}
and if $i \in I^{\text{im}}$, we have
\begin{equation*}
\begin{aligned}
(\alpha_i, \mu + 2 \rho) & = (\alpha_i, \lambda - \beta + 2 \rho)
= (\alpha_i, \lambda - \beta) + (\alpha_i, \alpha_i) \\
& = (\alpha_i, \lambda) - (\beta - \alpha_i, \alpha_i) \ge 0.
\end{aligned}
\end{equation*}

Therefore we obtain
\begin{equation*}
\begin{aligned}
\sum_{k=1}^{r} &  s_{k} (s_{k} - 1) (\alpha_{i_k}, \alpha_{i_k})
= (\mu + 2 \rho, \mu) - (\lambda+ 2 \rho, \lambda) \\
&  = (\mu - \lambda, \mu + \lambda + 2 \rho)
 = - \sum_{i \in I} a_{i} (\alpha_i, \lambda)
 - \sum_{i \in I} a_i (\alpha_i, \mu + 2 \rho) \\
 & = - \sum_{i \in I} a_i (\alpha_{i}, \lambda)
 - \sum_{i \in I^{\text{re}}} a_{i} (\alpha_i, \mu + 2 \rho)
 - \sum_{i \in I^{\text{im}}} a_{i} (\alpha_i, \mu + 2 \rho) \\
 & < - \sum_{i \in I} (\alpha_i, \lambda) - \sum_{i \in I^{\text{re}}} a_i
 - \sum_{i \in I^{\text{im}}} a_{i}(\alpha_{i}, \mu + \alpha_{i}) \\
 & = - \sum_{i \in I^{\text{re}}} a_{i} - \sum_{i \in I^{\text{im}}} a_{i}(\alpha_{i}, \lambda)
 - \sum_{i \in I^{\text{im}}} a_{i} (\alpha_{i}, \lambda - \sum_{j \in I^{\text{im}}} a_{j} \alpha_{j} + \alpha_{i}) \\
 & = - \sum_{i \in I^{\text{re}}} a_{i} - 2 \sum_{I^{\text{im}}} a_{i} (\alpha_i, \lambda)
 + \sum_{j \neq i} a_{i} a_{j} (\alpha_{i}, \alpha_{j})
 + \sum_{i \in I^{\text{im}}} a_{i} (a_{i} - 1) (\alpha_i, \alpha_i).
\end{aligned}
\end{equation*}

It follows that
\begin{equation} \label{eq:F}
\begin{aligned}
0 < &  - \sum_{i \in I^{\text{re}}} a_{i} - 2 \sum_{i \in I^{\text{im}}} a_{i} (\alpha_{i}, \lambda)
+ \sum_{i \in I^{\text{im}} \atop j \neq i}  a_{i} a_{j} (\alpha_i, \alpha_j) \\
& + \sum_{i \in I^{\text{im}}} (\alpha_i, \alpha_i)
\left(a_i (a_i -1) - \sum_{i_k = i} s_{k}(s_{k} - 1)\right).
\end{aligned}
 \end{equation}

 On the other hand, since $\sum_{i_k = i}  s_{k} = a_{i}$, we have
 $$a_{i} (a_{i} - 1) - \sum_{i_k = i} s_{k}(s_{k}-1) \ge 0$$
 where the equality holds if and only if there exists only one term in the sum.

 \vskip 2mm

 Since $(\alpha_i,\alpha_j) \le 0$ for $i \neq j$ and $(\alpha_i, \alpha_i) \le 0$ for $i \in I^{\text{im}}$,
 for $\beta = \sum_{k=1}^{r} s_{k} \alpha_{i_k}$, we conclude:

 \begin{itemize}

 \item[(a)] all $\alpha_{i_k}$ $(1 \le k  \le r)$ are imaginary simple roots,

 \vskip 2mm

 \item[(b)] $(\alpha_{i_k}, \alpha_{i_l}) = 0$ for all $1 \le k, l \le r$,

 \vskip 2mm

 \item[(c)] $(\alpha_{i_k}, \lambda) = 0$ for all $1 \le k \le r$,

 \vskip 2mm

 \item[(d)] if $i \notin I^{\text{iso}}$, then
 $\#\{k \mid i_k = i \} = 1.$

  \end{itemize}

 \vskip 2mm

 We will now give a character formula for the irreducible highest weight module $V(\lambda)$
 with $\lambda \in P^{+}$.
 As in \cite{Kac90}, we have
 \begin{equation} \label{eq:charA}
 (e^{\rho} R) \text{ch} V(\lambda) = \sum_{\mu \le \lambda} c(\mu) e^{\mu + \rho},
 \end{equation}
 where $c(\lambda) = 1$, $c(\mu) \in \Z$ and $c(\mu) = 0$ unless $\mu$ is a primitive weight.

 \vskip 2mm

 Applying $w \in W$ on \eqref{eq:charA},
thanks to Lemma \ref{lem:wR}, we obtain
\begin{equation*}
\begin{aligned}
\sum_{\mu \le \lambda} c(\mu) e^{\mu + \rho}
 & = \sum_{\mu \le \lambda} \epsilon(w) c(\mu) e^{w(\mu + \rho)} \\
& = \sum_{\mu + \rho \in C} \sum_{w \in W} \epsilon(w) c(\mu) w^{w(\mu + \rho)}.
\end{aligned}
\end{equation*}

 Set
 $$S := \sum_{\mu + \rho \in C} c(\mu) e^{\mu+\rho}.$$

 Then we have
 \begin{equation} \label{eq:Slambda}
 \sum_{\mu \le \lambda}  c(\mu) e^{\mu + \rho}
 = \sum_{w \in W} \epsilon(w) w(S).
  \end{equation}

 Therefore it suffices to compute $S$ to obtain the character formula.

 \vskip 2mm

 We define $E_{\lambda}$ to be the set of elements of the form
 $\alpha = \sum_{k=1}^{r} a_{k} \alpha_{i_k}$ $(a_{k} \in \Z_{>0})$
 satisfying the following conditions:

 \begin{itemize}

\item[(i)] all $\alpha_{i_k}$ are {\it even} imaginary simple roots for $1 \le k \le r$,

\vskip 2mm

\item[(ii)] $(\alpha_{i_k}, \alpha_{i_l}) = 0$ for all $1 \le k, l  \le r$,

\vskip 2mm

\item[(iii)] $(\alpha_{i_k}, \lambda) = 0$ for all $1 \le k \le r$.

 \end{itemize}

 \vskip 2mm

 For an element $\alpha = \sum_{k=1}^{r} a_{k} \alpha_{i_k}  \in E_{\lambda}$, we define
 \begin{equation} \label{eq:signE}
 \begin{aligned}
 & d_{i}(\alpha) = \begin{cases}
 \#\{k \mid i_k = i \} \ \ & \ \text{if} \ i \notin I^{\text{iso}}, \\
 \sum_{i_k = i} a_{k} \ \ & \ \text{if} \ i \in I^{\text{iso}},
 \end{cases} \\
 & \epsilon(\alpha) =
 \prod_{i \notin I^{\text{iso}}} (-1)^{d_{i}(\alpha)} \prod_{i \in I^{\text{iso}}} \phi(d_i(\alpha)),
 \end{aligned}
 \end{equation}
 where $\phi(n)$ is defined by $\prod_{k=1}^{\infty} (1 - q^k) = \sum_{n=0}^{\infty} \phi(n) q^n$.

 \vskip 2mm

 On the other hand, we define $O_{\lambda}$ to be the set of elements of the form
 $\beta = \sum_{l=1}^{s} b_{l} \alpha_{i_l}$ $(b_{l} \in \Z_{>0})$
  satisfying the following conditions:
 \begin{itemize}

\item[(i)] all $\alpha_{i_l}$ are {\it odd} imaginary simple roots for $1 \le l \le s$,

\vskip 2mm

\item[(ii)] $(\alpha_{i_k}, \alpha_{i_l}) = 0$ for all $1 \le k, l  \le s$,

\vskip 2mm

\item[(iii)] $(\alpha_{i_l}, \lambda) = 0$ for all $1 \le l \le s$.

 \end{itemize}

 \vskip 2mm

 For an element $\beta =  \sum_{l=1}^{s} b_{l} \alpha_{i_l}  \in O_{\lambda}$, we define
 \begin{equation} \label{eq:signO}
 \begin{aligned}
 & d_{i}(\beta) = \begin{cases}
 \#\{l \mid i_l = i \} \ \ & \ \text{if} \ i \notin I^{\text{iso}}, \\
 \sum_{i_l = i} b_{l} \ \ & \ \text{if} \ i \in I^{\text{iso}},
 \end{cases} \\
 & \epsilon(\beta) =
 \prod_{i \notin I^{\text{iso}}} (-1)^{d_{i}(\beta)} \prod_{i \in I^{\text{iso}}}
 c(\lambda - d_{i}( \beta) \alpha_{i}).
 \end{aligned}
 \end{equation}

\vskip 2mm

Finally, we define $F_{\lambda}$ to be the elements of the form $s = \alpha + \beta$
with $\alpha \in E_{\lambda}$, $\beta \in O_{\lambda}$ such that
$(\alpha, \beta) = 0$.
Set
$\epsilon(s) = \epsilon(\alpha) \epsilon(\beta)$,
and define
$$S_{\lambda} = \sum_{s \in F_{\lambda}} \epsilon(s) e^{-s}.$$

In the following theorem, we will prove $S = S_{\lambda}$,
which would give us the desired character formula.

\vskip 2mm

\begin{theorem} \label{thm:main}
{\rm
Let  $V(\lambda)$ be the irreducible highest weight module with
a dominant integral weight $\lambda \in P^{+}$.

\vskip 2mm

Then the character of $V(\lambda)$ is given by the formula
\begin{equation} \label{eq:main}
(e^{\rho} R) \text{ch} V(\lambda) = \sum_{w \in W} \epsilon(w) e^{w(\lambda + \rho)} w(S_{\lambda}).
\end{equation}
}
\end{theorem}

\begin{proof}
  Recall that
  $$(e^{\rho} R) \text{ch}\ V(\lambda) =\sum_{w \in W}\epsilon(w) w(S),$$
  where
  $$S = \sum_{\mu + \rho \in C} c(\mu) e^{\mu+\rho},$$
which can be rephrased as
  \begin{equation}\label{5}
  e^{\rho}\text{ch}\ V(\lambda)\prod_{\alpha\in\Delta_{0}^{+}}(1-e^{-\alpha})^{\text{dim}\ \g_{\alpha}}
  =\sum_{w\in W}\epsilon(w) w(S)\prod_{\alpha\in\Delta_{1}^{+}}(1+e^{-\alpha})^{\text{dim}\ \g_{\alpha}}.
  \end{equation}

  If $e^{\lambda-\sum l_{i}\alpha_{i}}$ is a term of $\text{ch}\ V(\lambda)$,\ then some $\alpha_{i}$ has nonzero inner product with $\lambda$.
  Hence the only terms of the form $e^{\lambda-\sum l_{i}\alpha_{i}+\rho}$ in  $S$
  are coming from $e^{\lambda+\rho}R$.

  \vskip 2mm

  Set $\lambda-\mu=\delta+\tau$,  $\delta=\sum s_{k}\alpha_{i_{k}}$, where $\alpha_{i_{k}}$ are even imaginary simple roots
and  $\tau=\sum_{k=1}^{n} c_{k}\beta_{k}+\sum_{t=1}^{m}l_{t}\gamma_{t}$ where  $\beta_{k}$ are odd nonisotropic imaginary simple roots, $\gamma_{t}$ are odd isotropic imaginary simple roots.

  \vskip 2mm

  {\bf Claim:} \
We have
   \begin{equation}\label{23}
   c(\mu)+\sum_{t=1}^{m}\sum_{l_{i_{t}}^{'}\le l_{i_{t}}}c(\mu+\sum_{k\leq t} l_{i_{k}}^{'}\gamma_{i_{k}})\prod_{k=1}^{t}\phi_{i_{k}}^{'}=\epsilon(\delta)(-1)^{n}\prod_{k=1}^{m}\phi_{\frac{l_{k}}{2}},
   \end{equation}
   where $\prod_{k\ge1}(1+q^{2k-1})=\sum \phi_{l}^{'}q^{l}$.

   \vskip 2mm

  To prove our claim,  we will use the induction on $n$.
   When $n=0$, the equality is correct.

  \vskip 2mm

  Without loss of generality,
  we may assume that $c_{k}$ are odd.
  For any given positive integer $N$, we assume that $n\ge 0$ and the equality holds for $n\textless N$.
 We would like to prove our claim when $n = N$.

    When  $n=N$, since the left-hand side of (\ref{5}) is even, one can see that
  \begin{align*}
    &c(\mu)+\sum_{k=1}^{n}c(\mu+\sum_{r\leq k}c_{j_{r}}\beta_{j_{r}})+\sum_{k=1}^{n}\sum_{t=1}^{m}\sum_{l_{i_{t}}^{'}\le l_{i_{t}}}c(\mu+\sum_{r\leq k}c_{j_{r}}\beta_{j_{r}}+\sum_{p\leq t}l_{i_{p}}^{'}\gamma_{i_{p}})\prod_{a=1}^{t}\phi_{i_{a}}^{'}\\
    &+\sum_{t=1}^{m}\sum_{l_{i_{t}}^{'}\le l_{i_{t}}}c(\mu+\sum_{p\leq t}l_{i_{p}}^{'}\gamma_{i_{p}})\prod_{k=1}^{t}\phi_{i_{k}}^{'}=0.
  \end{align*}
  By our induction hypothesis, we have
  \begin{align*}
  &\sum_{k=1}^{n}c(\mu+\sum_{r\leq k}c_{j_{r}}\beta_{j_{r}})+\sum_{k=1}^{n}\sum_{t=1}^{m}\sum_{l_{i_{t}}^{'}\le l_{i_{t}}}c(\mu+\sum_{r\leq k}c_{j_{r}}\beta_{j_{r}}+\sum_{p\leq t}l_{i_{p}}^{'}\gamma_{i_{p}})\prod_{a=1}^{t}\phi_{i_{a}}^{'}\\
  &=\epsilon(\delta)\sum_{k=1}^{n} \binom{n}{k} (-1)^{n-k}\prod \phi_{\frac{l_{t}}{2}}=\epsilon(\delta)(-1)^{n+1}\prod \phi_{\frac{l_{t}}{2}}.
  \end{align*}
  Hence (\ref{23}) is correct.

  \vskip 2mm

  Next, we will show that
  $$c(\mu)=\epsilon(\delta)(-1)^{n}c(\lambda-\sum l_{t}\gamma_{t}).$$

  Set $\gamma=\sum l_{t}\gamma_{t}$.
  When $\text{ht}\ (\gamma)=m$,  all $l_{t} = 1$. Hence by \cite{Ray95}, we have
  $$c(\mu)=\epsilon(\delta) (-1)^{n+m}=\epsilon(\delta)(-1)^{n}c(\lambda-\sum \gamma_{t}).$$
  Given a positive integer $N$, if $\text{ht}\ (\gamma)\textless N$, assume that
  $$c(\mu)=\epsilon(\delta)(-1)^{n}c(\lambda-\gamma).$$
    When $\text{ht}\ (\gamma)=N$, thanks to (\ref{23}) and our induction
  hypothesis, a direct calculation yields
  $$c(\lambda-\gamma)+\sum_{t=1}^{m}\sum_{l_{i_{t}}^{'}\le l_{i_{t}}}c(\lambda-\gamma+\sum_{k\leq t} l_{i_{k}}^{'}\gamma_{i_{k}})\prod_{k=1}^{t}\phi_{i_{k}}^{'}=\prod_{k=1}^{m}\phi_{\frac{l_{k}}{2}}$$
 $$c(\mu)+\epsilon(\delta)(-1)^{n}\sum_{t=1}^{m}\sum_{l_{i_{t}}^{'}\le l_{i_{t}}}c(\lambda-\gamma+\sum_{k\leq t} l_{i_{k}}^{'}\gamma_{i_{k}})\prod_{k=1}^{t}\phi_{i_{k}}^{'}=\epsilon(\delta)(-1)^{n}\prod_{k=1}^{m}\phi_{\frac{l_{k}}{2}}.$$
 Hence
  $$c(\mu)=\epsilon(\delta)(-1)^{n}c(\lambda-\gamma).$$

  Finally, we will check
  $$c(\lambda-\gamma)=\prod_{t}c(\lambda-l_{t}\gamma_{t}).$$

  If $\text{ht}\ (\gamma)=m$, the equality holds because $c(\lambda-\gamma)=(-1)^{m}$.

\vskip 2mm

Suppose that $\text{ht} (\gamma)=N$. Then
we have $ c(\lambda-\gamma)=\prod_{t=1}^{m}c(\lambda-l_{t}\gamma_{t})$.
When $\text{ht} (\gamma)=N+1$,  let $\gamma=\sum_{t\ge2}l_{t}\gamma_{t}+(l_{1}+1)\gamma_{1},$ and $\mu=\lambda-\gamma$.
From (\ref{23}), we can deduce
  \begin{align*}
  c(\mu) + &\sum_{t\le m, 2\le i_{p}}\sum_{l^{'}\le  l}c(\mu+\sum_{p\leq t}l_{i_{p}}^{'}\gamma_{i_{p}})\prod_{k=1}^{t}\phi_{l_{i_{k}}^{'}}^{'}+\sum_{t\le m, 2\le i_{p}}\sum_{l^{'}\le  l}c(\mu+l_{1}^{'}\gamma_{1} \\
  &+\sum_{p\leq t}l_{i_{p}}^{'}\gamma_{i_{p}})\prod_{k=1}^{t} \phi_{l_{i_{k}}^{'}}^{'}\phi_{l_{1}^{'}}^{'}
 +\sum_{l_{'}\le l_{1}}c(\mu+l_{1}^{'}\gamma_{1})\phi_{l_{1}^{'}}^{'}  \\
  &=\prod_{t=1}^{m}\phi_{\frac{l_{t}}{2}}.
  \end{align*}

  Our induction hypothesis gives
  \begin{align*}
     {\rm (i)} & \sum_{t\le m, 2\le i_{p}}\sum_{l^{'}\le  l}c(\mu+\sum_{p\leq t}l_{i_{p}}^{'}\gamma_{i_{p}})\prod_{k=1}^{t}\phi_{l_{i_{k}}^{'}}^{'} \\
    & =  c(\lambda-(l+1)\gamma_{1})\sum_{t\le m, 2\le i_{t}}\sum_{l^{'}\le  l}c(\lambda-\sum_{k\ge 2}l_{k}\gamma_{k}+\sum_{p\leq t}l_{i_{p}}^{'}\gamma_{i_{p}})\prod_{k=1}^{t}\phi_{l_{i_{k}}^{'}}^{'}, \\
    {\rm (ii)} &  \sum_{t\le m, 2\le i_{p}}\sum_{l^{'}\le  l}c(\mu+l_{1}^{'}\gamma_{1}+\sum_{p\leq t}l_{i_{p}}^{'}\gamma_{i_{p}})\prod_{k=1}^{t} \phi_{l_{i_{k}}^{'}}^{'}\phi_{l_{1}^{'}}^{'}\\
   & = \sum_{l_{1}^{'}\le l_{1}+1}c(\lambda-(l_{1}+l-l_{1}^{'})\gamma_{1})\phi_{l_{1}^{'}}\sum_{t\le m, 2\le i_{p}}\sum_{l^{'}\le  l}c(\lambda-\sum_{t\ge 2}l_{t}\gamma_{t}+\sum_{p\leq t}l_{i_{p}}^{'}\gamma_{i_{p}})\prod_{k=1}^{t}\phi_{l_{i_{k}}^{'}}^{'},\\
     {\rm (iii)}  & \ \ \   \sum_{l_{'}\le l_{1}}c(\mu+l_{1}^{'}\gamma_{1})\phi_{l_{1}^{'}}^{'}=\sum_{l_{1}^{'}\le l_{1}}c(\lambda-(l_{1}+1-l_{1}^{'})\gamma_{1})\phi_{l_{1}^{'}}c(\lambda-\sum_{t\ge 2}l_{t}\gamma_{t}).
  \end{align*}

  Hence we obtain
  \begin{align*}
  &\sum_{t\le m, 2\le i_{p}}\sum_{l^{'}\le  l}c(\mu+\sum_{p\leq t}l_{i_{p}}^{'}\gamma_{i_{p}})\prod_{k=1}^{t}\phi_{l_{i_{k}}^{'}}^{'}+\sum_{t\le m, 2\le i_{p}}\sum_{l^{'}\le  l}c(\mu+l_{1}^{'}\gamma_{1}+\sum_{p\leq t}l_{i_{p}}^{'}\gamma_{i_{p}})\prod_{k=1}^{t} \phi_{l_{i_{k}}^{'}}^{'}\phi_{l_{1}^{'}}^{'} \\
  &=\phi_{\frac{l+1}{2}}\sum_{t\le m, 2\le i_{p}}\sum_{l^{'}\le  l}c(\lambda-\sum_{t\ge 2}l_{t}\gamma_{t}+\sum_{p\leq t}l_{i_{p}}^{'}\gamma_{i_{p}})\prod_{k=1}^{t}\phi_{l_{i_{k}}^{'}}^{'}.
  \end{align*}
 It follows that
$$c(\mu)=c(\lambda-(l_{1}+1)\gamma_{1})\prod_{t\ge 2}c(\lambda-l_{t}\gamma_{t}),$$

  Thus   we coclude
  $$c(\mu)=\epsilon(\delta)(-1)^{n}\prod_{t=1}^{m} c(\lambda-l_{t}\gamma_{t}),$$
  which proves $S = S_{\lambda}$.
 \end{proof}


\vskip 7mm

\begin{thebibliography}{10}


 \bibitem{Bor88} R. E. Borcherds,
 {\em Generalized Kac-Moody algebras},
J. Algebra {\bf 115} (1988), 501--512.

 \bibitem{Bor92} R. E.Borcherds,
 {\em Monstrous moonshine and monstrous Lie superalgebras},
 Inv. Math. {\bf 109}, (1992), 405-444.

\bibitem{Bozec2014a} T. Bozec,
{\em Quivers with loops and perverse sheaves},
Math. Ann. {\bf 362} (2015), 773--797.

\bibitem{Bozec2014b} T. Bozec,
{\em Quivers with loops and generalized crystals},
Compositio Math. {\bf 152} (2016), 1999--2040.

\bibitem{BSV2016} T. Bozec, O. Schiffmann, E. Vasserot,
{\em On the number of points of nilpotent quiver varieties over finite fields},
Ann. Sci. Ec. Norm. Super. {\bf 53} (2020) 1501--1544.

\bibitem{Humph72}
J. E. Humphreys,
{\em Introduction to Lie Algebras and Representation Theory},
Grad. Texts in Math. {\bf 9}, Springer-Verlag, New York, 1972.

\bibitem{Kac68} V. G. Kac,
{\em Simple irreducible graded Lie algebras of finite growth},
Math. USSR-Izvestija {\bf 2} (1968), 1271--1311.

\bibitem{Kac74} V. G. Kac,
{\em Infninite dimensional Lie algebras and Dedekind's $\eta$-function},
Funct. Anal. Appl. {\bf 8} (1974), 68--70.


\bibitem{Kac90} V. G. Kac,
{\em Infinite-dimensional Lie Algebras},
Cambridge University Press, 1990.

\bibitem{Kac77} V. G. Kac,
{\em Lie superalgebras},
Adv. Math. {\bf 26} (1977), 8--96.

\bibitem{Kac78} V. G. Kac,
{\em Infinite-dimensional algebras, Dedekind's $\eta$-function,
classical Mobius function and the very strange formula},
 Adv. Math. {\bf 30} (1978), 85--136.


\bibitem{Moody68} R. V. Moody,
{\em A new class of Lie algebras},
J. Algebra{\bf 10} (1968), 211--230.


\bibitem{Ray95} U. Ray,
{\em  A character formula for generalized Kac-Moody superalgebras},
J. Algebra {\bf 177} (1995), 154--163.




\end{thebibliography}
\end{document}